\documentclass[11pt,twoside]{article}
\usepackage{amssymb}
\usepackage[mathscr]{eucal}
\usepackage{eufrak}
\usepackage{amsmath,amsthm}
\usepackage{mathrsfs}
\usepackage[colorlinks=true,
   urlcolor=blue,           filecolor=green,      
   citecolor=green,      
   linkcolor=red,           bookmarks=true,
  unicode,
   plainpages=false,   ]{hyperref}
\usepackage{color}
\def\bp{\begin{proof}}
\def\ep{\end{proof}}
\def\n{\nabla}

\def\intl#1{\int\limits_{#1}}
\def\intll#1#2{\int\limits_{#1}^{#2}}
\def\dm{|\hskip-0.05cm|}

\def\OO{\Omega}
\def\displ{\displaystyle}

\def\VS{\vspace{6pt}\\\displ }
\def\rf#1{{\rm(\ref{#1})}}

\def\R{\Bbb R}
\def\N{\Bbb N}

\def\à {\`{a}}

\def\dy{\displaystyle}
\def\vep{\varepsilon}

\def\be{\begin{equation}}
\def\ba{\begin{array}}
\def\ea{\end{array}}
\def\ee{\end{equation}}
\def\vs1{\vspace{1ex}}

\def\ov{\overline}

\def\po{\partial\Omega}

\def\é{\'{e}}
\font\sc=cmcsc10
\title{\large\bf Navier-Stokes equations:\\a new estimate of a possible gap related to  the energy equality of a suitable weak solution}
\author{\sc  Paolo Maremonti$^1$, Francesca Crispo$^1$ and  Carlo Romano Grisanti$^2$ 
\thanks{$^1$ Dipartimento di Matematica e Fisica,  
Universit\`{a} degli 
Studi della Campania
``L. Vanvitelli'', via Vivaldi 43, 81100 Caserta,
 Italy. francesca.crispo@unicampania.it\,,
paolo.maremonti@unicampania.it      \newline $^2$ Dipartimento di Matematica Universit\`{a} di Pisa, via Buonarroti 1/c, 56127 Pisa, Italy. carlo.romano.grisanti@unipi.it}}
\date{}
\begin{document}
\markboth{\footnotesize\rm  P.Maremonti, F.Crispo and  C.R. Grisanti} {\footnotesize\rm
Navier-Stokes equations: a new estimate of a possible gap...}
\maketitle 
\noindent
\newcommand{\red}{\protect\bf}
\renewcommand\refname{\centerline
{\red {\normalsize \bf References}}}
\newtheorem{ass}
{\bf Assumption} 
\newtheorem{defi}
{\bf Definition} 
\newtheorem{tho}
{\bf Theorem} 
\newtheorem{rem}
{\sc Remark} 
\newtheorem{lemma}
{\bf Lemma} 
\newtheorem{coro}
{\bf Corollary} 
\newtheorem{prop}
{\bf Proposition} 
\renewcommand{\theequation}{\arabic{equation}}
\setcounter{section}{0}
\par{\bf Abstract} - The paper is concerned with the IBVP of the Navier-Stokes equations. The result of the paper is in the wake of analogous results obtained by the authors in previous articles \cite{CGM-I,CGM-II}. The goal is to estimaste the possible gap between the energy equality and the energy inequality deduced for a weak solution. 
\vskip1cm\section{\large Introduction}This note concerns the 3D-Navier-Stokes initial boundary 
value problem:
\be\label{NS}\ba{l}v_t+v\cdot
\nabla v+\nabla\pi_v=\Delta
v+f,\;\nabla\cdot
v=0,\mbox{ in }(0,T)\times\OO,\VS v=0\mbox{ on }(0,T)\times\po,\hskip0.12cm
v(0,x)=v_0(x)\mbox{ on
}\{0\}\times\OO.\ea\ee  In system \rf{NS} $\OO\subseteq\R^3$ is assumed bounded or exterior, and its boundary is assumed  smooth.
\par In the two recent papers \cite{CGM-I,CGM-II} the authors look for an energy equality for suitable weak solutions. Here,  the term suitable is meant in the sense that a new solution is exhibited and not that an improvement is { obtained to} the one given in \cite{CKN}. Actually, the crucial result of papers \cite{CGM-I,CGM-II}, and it seems the first, is the strong convergence in  $L^p(0,T;W^{1,2}(\OO))\cap L^2(0,T;L^2(\OO))$, for all $T>0$ and $p\in[1,2)$, of a sequence $\{v^m\}$ of smooth solutions to the ``Leray's approximating Navier-Stokes Cauchy problem'' (see \rf{MNS} below), \cite{L}. \par Since the strong convergence is not in $L^2(0,T;W^{1,2}(\OO))$, the authors attempt to obtain the energy equality employing the (differential and integral) energy equality of the approximating solutions and   some auxiliary functions. Actually, the approaches used so far allow to prove an energy equality which involves other quantities.  Here  it is proved that a suitable weak solution exists and  satisfies the following relation  
\be\label{SLM-O}\dm v(t)\dm_2^2+2 \intll st\dm\n v(\tau)\dm_2^2d\tau  -\dm v(s)\dm_2^2-\intll st(f,v)d\tau=-M(s,t)\mbox{ for all }0<s<t\in\mathcal T\,,\ee
where  $$\mathcal T:=\left\{t\in(0,T): \|v^m(t)\|_{1,2}\to\|v(t)\|_{1,2}\right\}$$ is of full measure in $(0,T)$,  and
$$-M(s,t):=-2\lim_{\alpha\to1^-}\overline{\lim_m}\intl{J^m(\alpha)}\dm \n v^m(\tau)\dm_2^2\,d\tau=\lim_{\alpha\to1^-}\overline{\lim_m}\mbox{\large$\underset{h\in\N(\alpha,m)}\sum$}\Big[ \dm v^m(t_h)\dm_2^2-\dm v^m(s_h)\dm_2^2\Big]\,$$
where $J^m(\alpha)$ is the union of, at most, a countable sequence ($\N(\alpha,m)$) of disjoint intervals $(s_h,t_h)\subset (s,t)$ and  the following holds: $$\lim_{\alpha\to1^{-}}\frac{|J^m(\alpha)|}{1-\alpha}\leq\frac1\pi\dm v_0\dm_2^2+\frac2\pi\intll0t(f,v)d\tau\,,\;\mbox{ uniformly in } m\in\N.$$ Instead in the case of $s=0$, one obtains
\be\label{SLM-I}\dm v(t)\dm_2^2+2 \intll 0t\dm\n v(\tau)\dm_2^2d\tau  -\dm v_0\dm_2^2-\intll 0t(f,v)d\tau=-M(0,t)\mbox{ for all }t\in\mathcal T\,,\ee
where $$-M(0,T):=\lim_{s_k\to0}-M(s_k,t)\,,\mbox{ for any }\{s_k\}\subset \mathcal T\,.$$
  Roughly speaking the above intervals seem to contain the possible singular points $S$ of the weak solution that, as is known, has $\mathcal H^\frac12(S)=0$ ($\mathcal H^a$ Hausdorff's measure), \cite{Sh}.
Of course, independently of the   meaning of the conjecture  for the intervals, from a physical view point the energy relation \rf{SLM-O} would  add a dissipative quantity which is not justifiable. If this  is a necessary consequence of an initial datum  only in $L^2$, then from a physical point of view it is a right reason to reject the $L^2$-class as a class of existence.  \par However,   the  validity of an energy equality, without requiring extra conditions\footnote{\, In this connection  in paper \cite{Mi}, the so called Prodi-Serrin condition for the energy equality for a weak solution
  is not required  on the whole  interval of existence, but just on $(\vep,T)$, that is $L^4(\vep,T;L^4(\OO))$,  for all $\vep>0$. This means that no extra assumption on the initial datum in $L^2$ is needed  for the validity of the energy equality. \par Following \cite{GPG}, under the same quoted assumption, the same result of energy equality holds in the set of very-weak solutions.}, is interesting to better delimit the case of validity of possible counterexamples. \par Actually, in the papers \cite{BV} and \cite{ABC}  two examples of non-uniqueness are furnished. \par  The former works for very-weak solutions, which are continuous in $L^2$-norm, but  do not verify an energy inequality of the kind given by Leray-Hopf, in other words neglecting the term $M(s,t)$ with $\leq0$. Further, in the case of Leray-Hopf weak solutions their counterexample does not work. \par  The latter works with a homogeneous initial datum. Actually, the non-uniqueness is exhibited for solutions corresponding to a suitable data force, that, among other things, allows an energy equality.
\vskip0.1cm The plan of the paper is the following. In sect. 2  some preliminary lemmas are recalled and some new results of strong convergence are furnished. In sect. 3  the statement and  the proof of the chief result are performed.
\section{\large Preliminary results}
We set $J^{1,2}(\OO)$:=completion of $\mathscr C_0(\OO)$ in $W^{1,2}$-norm, where $\mathscr C_0(\OO)$ is the set of the test functions of the hydrodynamics.
\begin{defi}\label{WS}{\sl For weak solution to the IBVP \rf{NS} we mean a field $v:(0,\infty)\time\OO\to \R^3$ such that for all $T>0$
\begin{itemize}\item[1.] $v\in L^\infty(0,T;L^2(\OO))\cap L^2(0,T;J^{1,2}(\OO))\,,$\item[2.] the field $v$ solves the integral equation \newline $\displ\intll
st\Big[(v,\varphi_\tau)-(\nabla
v,\nabla
\varphi)+(v\cdot\nabla\varphi,v)+(\pi_v,\nabla
\cdot\varphi)\Big]d\tau+(v(s),\varphi
(s))=(v(t),\varphi(t)),$ \newline\null\hskip5cm for all  $\varphi\in C^1_0([0,T)\times\OO ),$\item[3.]$\displ \lim_{t\to0}\dm v(t)-v_0\dm_2=0\,.$\end{itemize}
}\end{defi}\vskip0.1cm
For our goals we consider  a mollified Navier-Stokes system. Hence problem \rf{NS} becomes
\be\label{MNS}\ba{l}v_t^m+J_m[v^m]\cdot
\nabla v^m+\nabla\pi_{v^m}=\Delta
v^m+f,\;\nabla\cdot
v^m=0,\mbox{ in }(0,T)\times\OO,\VS v^m=0\mbox{ on }(0,T)\times\po,\hskip0.12cm
v^m(0,x)=v_0^m(x)\mbox{ on
}\{0\}\times\OO,\ea\ee
where $f\in L^2(0,T,L^2(\Omega))$, \ $J_m[\cdot]$ is a mollifier and  $\{v_0^m\}\subset J^{1,2}(\OO)$ converges to $v_0$ in $J^2(\OO)$.
\begin{lemma}\label{EXR}{\sl For all $m\in\N$ there exists a unique solution to problem \rf{MNS} such that for all $T>0$ \be\label{EXR-I}\ba{l}v^m\in C([0,T);J^{1,2}(\OO))\cap L^2(0,T;W^{2,2}(\OO))\,,\VS
v^m_t,\n\pi^m\in L^2(0,T;L^2(\OO))\,.\ea\ee Moreover, the sequence $\{v^m\}$ is strong convergent to a limit $v$ in $L^p(0,T;W^{1,2}(\OO))\cap L^2(0,T;L^2(\OO))$, for all $p\in[1,2)$, and the limit $v$ is a weak solution to problem \rf{NS} with $(v(t),\varphi)\in C([0,T))$, for all $\varphi\in J^2(\OO)$.}\end{lemma}
 \bp This lemma for data force $f=0$ is  Theorem\,6.1.1 proved in \cite{CGM-I}. It is not difficult to image that the proof can be modified without difficulty assuming $f\ne0$. So that we consider as achieved the proof of the lemma.\ep
\begin{lemma}\label{TINT}{\sl Let $\OO\subseteq\R^n$ and  let   $u\in W^{2,2}(\OO)\cap J^{1,2}(\OO) $. Then there exists a constant $c$ independent of $u$ such that    
\be\label{INTIII}\dm u\dm_r\!\leq\! c\dm P\Delta u\dm_2^a\dm u\dm_q^{1-a},\quad a\big(\mbox{$\frac12-\frac2n$}\big)+(1-a)\mbox{$\frac1q$}={\textstyle\frac1r},\ee
provided that   $a\in [0,1)$.}\end{lemma}\bp See \cite{MRIII,MRIV}\,.\ep
\begin{lemma}\label{CL}{\sl For any $T>0$, there exists a constant $M>0$, not depending on $m$, such that
$$\int_0^T\frac{\left|\frac d{dt}\|v^m(t)\|_2^2\right|}{\left(1+\|\nabla v^m\|_2^2\right)^2}\,dt\le M$$
where $v^m$ is the solution of problem \eqref{MNS} stated in Lemma\,\ref{EXR}.
}
\end{lemma}
\begin{proof} 
By virtue of the regularity of $(v^m,\pi^m)$ stated in \rf{EXR-I}, we multiply equation \rf{MNS}$_1$ by $P\Delta v^m-v^m_t$. Integrating by parts on $\OO$, and applying the H\"older inequality, we get 
\be\label{PDelta-vt}\dm P\Delta v^m-v^m_t\dm_2^2\leq2\dm v^m\cdot\n v^m\dm_2^2+2\|f\|_2^2\,,\mbox{ a.e. in }t>0\,.\ee
Applying inequality \rf{INTIII}
with $r=\infty$ and $q=6$, by virtue of the Sobolev inequality, we obtain 
\be\label{NLT}\dm v^m\cdot\n v^m\dm_2\leq \dm v^m\dm_\infty\dm \n v^m\dm_2\leq c\dm  P\Delta v^m\dm_2^\frac12\dm \n v^m\dm_2^\frac32.\ee  
By inequalities \eqref{PDelta-vt} and \eqref{NLT}, we get 
\be\label{PdeltaL6}\dy\vs1\dm P\Delta v^m-v^m_t\dm_2^2\leq c\dm  P\Delta v^m\dm_2\dm \n v^m\dm_2^3+2\|f\|_2^2
\hfill\dy\le\frac12\|P\Delta v^m\|_2^2+
c\dm \n v^m\dm_2^6+2\|f\|_2^2,\ee
for all $m\in\N$ and a.e. in $t>0\,$.
Substituting in inequality \eqref{PdeltaL6} the identity
\be\label{ddtnablavm-I}\dy\vs1\mbox{$\frac d{dt}$}\dm \n v^m\dm_2^2+\dm P\Delta v^m\dm_2^2+\dm v^m_t\dm_2^2=\dm P\Delta v^m-v^m_t\dm_2^2\ee
and dividing by $(1+\dm \n v^m(t)\dm_2^2)^2$, we get the following estimate
$$\frac{\mbox{$\frac d{dt}$}\dm \n v^m\dm_2^2}{(1+\dm \n v^m\dm_2^2)^2\hskip-0.1cm\null}\hskip0.1cm+\frac{\frac12\dm P\Delta v^m\dm_2^2+\dm v_t^m\dm_2^2}{(1+ \dm \n v^m\dm_2^2)^2}\leq c\dm \n v^m\dm_2^2+\frac{2\|f\|_2^2}{\left(1+\|\nabla v^m\|_2^2\right)^2}\,.$$ 
Integrating on $(0,T)$ we have
\begin{equation*}\begin{split}\frac1{1+ \|\nabla v^m_0\|_2^2}-\frac1{1+ \|\nabla v^m(T)\|^2_2}+\int\limits_0^T\frac{\frac12\dm P\Delta v^m\dm_2^2+\dm v_t^m\dm_2^2}{(1+ \dm \n v^m\dm_2^2)^2}\,dt\\
\le c \int\limits_0^T\|\nabla v^m\|_2^2\,dt+2\int_0^T\frac{2\|f\|_2^2}{\left(1+\|\nabla v^m\|_2^2\right)^2}\,dt\le C.\end{split}\end{equation*}
It follows that
$$\int\limits_0^T\frac{\dm P\Delta v^m\dm_2^2}{(1+ \dm \n v^m\dm_2^2)^2}\,dt\le2C+2, \quad \int\limits_0^T\frac{\dm v_t^m\dm_2^2}{(1+ \dm \n v^m\dm_2^2)^2}\,dt\le C+1.$$
Using the identity \eqref{ddtnablavm-I} we get
\begin{equation*}\begin{split}\int\limits_0^T\frac{\dm P\Delta v^m-v^m_t\dm_2^2}{(1+ \dm \n v^m\dm_2^2)^2}\,dt=\int\limits_0^T\frac{\mbox{$\frac d{dt}$}\dm \n v^m\dm_2^2}{(1+\dm \n v^m\dm_2^2)^2}\,dt+\int\limits_0^T\frac{\dm P\Delta v^m\dm_2^2}{(1+ \dm \n v^m\dm_2^2)^2}\,dt+ \int\limits_0^T\frac{\dm v_t^m\dm_2^2}{(1+ \dm \n v^m\dm_2^2)^2}\,dt\\
\le-\frac1{1+ \|\nabla v^m(T)\|^2_2}+\frac1{1+ \|\nabla v^m_0\|_2^2}+3C+3\le3C+4.
\end{split}\end{equation*}
Using once again identity \eqref{ddtnablavm-I} we get
\begin{equation*}\begin{split}\int\limits_0^T\frac{\left|\mbox{$\frac d{dt}$}\dm \n v^m\dm_2^2\right|}{(1+\dm \n v^m\dm_2^2)^2}\,dt
\le\int\limits_0^T\frac{\dm P\Delta v^m-v^m_t\dm_2^2}{(1+ \dm \n v^m\dm_2^2)^2}\,dt+\int\limits_0^T\frac{\dm P\Delta v^m\dm_2^2}{(1+ \dm \n v^m\dm_2^2)^2}\,dt+\int\limits_0^T\frac{\dm v_t^m\dm_2^2}{(1+ \dm \n v^m\dm_2^2)^2}\,dt\\
\le6C+7=:M.\end{split}\end{equation*}  
\end{proof}

\begin{lemma}\label{WC}{\sl Let $\{h_m(t)\}$ be a sequence of non-negative functions   bounded in $L^1(0,T)$. Also, assume that $h_m(t)\to h(t)$ a.e. in $t\in (0,T)$ with $h(t)\in L^1(0,T)$. Let be $g:(0,\alpha_0)\longrightarrow\R$ a continuous and strictly increasing function such that $\lim\limits_{\alpha\to\alpha_0}g(\alpha)=+\infty$ and $p:[0,1)\times\R\longrightarrow[0,1]$ a continuous function such that
$p(\alpha,\rho)=1$ if $0\le\rho\le g(\alpha)$, $p(\alpha,\cdot)$ is weakly decreasing and $\lim\limits_{\rho\to+\infty}p(\alpha,\rho)=0$ for any $\alpha\in(0,\alpha_0)$.

Then we get
\be\label{WC-I}\lim_{\alpha\to\alpha_0}\lim_m\intll0Th_m(t)p(\alpha,h_m(t))dt=\intll0Th(t)dt\,,\ee}
\end{lemma}
\bp 
We have
$$\ba{ll}\displ\intll0Th_m(t)p(\alpha,h_m(t))dt\hskip-0.2cm&\displ=\intll0T(h_m(t)-h(t))p(\alpha,h_m(t))dt+
\intll0Th(t)p(\alpha,h_m(t))\\&=:I_1(\alpha,m)+I_2(\alpha,m)\,.\ea$$ 
We fix $\alpha\in(0,\alpha_0)$ and we consider the first integral. For any $\vep\in(0,\alpha_0-\alpha)$ we set
\begin{equation}\label{Jm-+}J^-_m(\vep)=\{t:h_m(t)\leq g(\alpha_0-\vep)\},\qquad J^+_m(\vep)=\{t:g(\alpha_0-\vep)<h_m(t)\}.\end{equation}
Hence we have
\begin{equation*}\begin{split}I_1(\alpha,m)=\intll0T\chi_{J^-_m(\vep)}(t)(h_m(t)-h(t))p(\alpha,h_m(t)) dt+\intll0T\chi_{J_m^+(\vep)}(t)(h_m(t)-h(t))p(\alpha,h_m(t))dt\,\\
=:I_1^-(\alpha,m,\vep)+I_1^+(\alpha,m,\vep).\end{split}\end{equation*}
By \eqref{Jm-+} we get
$$|\chi_{J^-_m(\vep)}(t)(h_m(t)-h(t))p(\alpha,h_m(t))|\leq g(\alpha_0-\vep)+|h(t)|$$
hence, by the dominated convergence theorem, we have
\be\label{limI1-}\lim_m I_1^-(\alpha,m,\vep)=0, \qquad\forall\,\alpha,\vep.\ee
Since $p(\alpha,\cdot)$ is decreasing, we get
$$\left|\chi_{J_m^+(\vep)}(t)(h_m(t)-h(t))p(\alpha,h_m(t))\right|\le p(\alpha,g(\alpha_o-\vep))\left(|h_m(t)|+|h(t)|\right).$$
Using the boundedness of the sequence $(h_m)$ in $L^1$ we obtain that
\be\label{boundI1+}|I_1^+(\alpha,m,\vep)|\le c p(\alpha,g(\alpha_0-\vep)),\qquad\forall m\in\N.\ee
By \eqref{limI1-} and \eqref{boundI1+} we get
$$0\le\limsup_m |I_1(\alpha,m)|\le c p(\alpha,g(\alpha_0-\vep)),\qquad\forall\,\alpha,\vep.$$
Since $\lim\limits_{\vep\to0}p(\alpha,g(\alpha_0-\vep)=0$ we have that
$$\lim_m I_1(\alpha,m)=0,\qquad\forall\,\alpha.$$
Now we consider the integral $I_2(\alpha,m)$. Since $\left|p(\alpha,h_m(t))h(t)\right|\leq1$ and $\lim\limits_m h_m(t)=h(t)$ a.e. in $t\in(0,T)$, by the dominated convergence theorem, we get
$$\lim_mI_2(\alpha,m)=\intll0Th(t)p(\alpha,h(t))dt.$$
Finally, since $\lim\limits_{\alpha\to\alpha_0}p(\alpha,h(t))=1$ we have that
$$\lim_{\alpha\to\alpha_0}\lim_mI_2(\alpha,m)=\intll0Th(t)dt,$$
and this completes the proof.
\ep
\section{\large The chief result}
We recall the definition $$\mathcal T:=\left\{t\in(0,T): \|v^m(t)\|_{1,2}\to\|v(t)\|_{1,2}\right\},$$ where $\{v^m\}$ is the sequence of solutions to problem \rf{MNS}. By virtue of the strong convergence stated in Lemma\,\ref{EXR}, the set $\mathcal T$ is certainly not empty and, as matter of  fact, it is of full measure in $(0,T)$ .
\begin{tho}\label{EE}{\sl Let $v$ be the weak solution stated in Lemma\,\ref{EXR}. Then
 $v$ satisfies the relation  
 \be\label{SLM}\dm v(t)\dm_2^2+2 \intll st\dm\n v(\tau)\dm_2^2d\tau  -\dm v(s)\dm_2^2-\intll st(f,v)d\tau=-M(s,t)\,,\mbox{ for all }t>s>\in \mathcal T,\ee
 with 
$$-M(s,t):=-2\lim_{\alpha\to1^-}\overline{\lim_m}\intl{J^m(\alpha)}\dm \n v^m(\tau)\dm_2^2\,d\tau=\lim_{\alpha\to1^-}\overline{\lim_m}\mbox{\large$\underset{h\in\N(\alpha,m)}\sum$}\Big[ \dm v^m(t_h)\dm_2^2-\dm v^m(s_h)\dm_2^2\Big]$$
where $J^m(\alpha)={\underset {i\in \N(\alpha,m)}\cup}(s_i,t_i)$ with $\N(\alpha,m)\subseteq\N$ which is, at most, a sequence of integers with   $(s_i,t_i)\cap (s_j,t_j)=\emptyset$ for any $i\not=j$ and \be\label{LMJ}\displ\lim_{\alpha\to1^-}\frac{|J^m(\alpha)|}{1-\alpha}\leq \frac1\pi\dm v_0\dm_2^2+\frac2\pi\intll0t(f,v)d\tau\,,\mbox{ uniformly with respect to }m\,.\ee
Moreover, if $s=0$, the relation \eqref{SLM} holds  setting $s=0$ in the left-hand side, and   with the right-hand side replaced by $\lim\limits_k M(s_k,t)$ where $\{s_k\}$ is any sequence in $\mathcal T$ converging to $0$.
}
\end{tho}
\bp We consider the sequence $\{v^m\}$ of solutions to problem \rf{MNS} whose existence is ensured by Lemma\,\ref{EXR}. For all $m\in\N$ the Reynolds-Orr equation holds:
\be\label{ER}\frac d{d\tau}\dm v^m(\tau)\dm_2^2+2\dm \n v^m(\tau)\dm_2^2=(f,v^m)\,.\ee
We set $\rho_m(t):=\dm\n v^m(t)\dm_2^2$\,,
and we consider \be\label{WF}\alpha\in(0,1)\,,\;p(\alpha,\rho_m):=\left\{\ba{ll}1&\mbox{if }\rho_m\in [0,\tan\alpha\frac\pi2]\vspace{4pt}\\ \frac{\frac\pi2-\arctan\rho_m}{(1-\alpha)\frac\pi2}&\mbox{if }\rho_m\in(\tan\alpha\frac\pi2\,,\infty)\,\ea\right..\ee
Let be 
$$\mathcal T=\left\{t\in(0,T): \| v^m(t)\|_{1,2}\to\| v(t)\|_{1,2}\right\}$$ 
and fix $s,t\in \mathcal T$, with $s<t$\,. Let $\alpha_1$ be  such that 
$$\max\{\dm \n v(s)\dm_2^2,\dm\n v(t)\dm_2^2\}<\tan\alpha\frac\pi2\,,\mbox{ for all }\alpha\in(\alpha_1,1)\,.$$
 Hence, by virtue of the pointwise  convergence, we claim the existence of $m_0$ such that\be\label{MV}\max\{\dm \n v^m(s)\dm_2^2,\dm\n v^m(t)\dm_2^2\}<\tan\alpha\frac\pi2\,, \mbox{ for all }m\geq m_0\mbox{ and }\alpha\in(\alpha_1,1)\,.\ee We set $\displ A^m:=\max_{[s,t]}\rho_m(t)$. We denote by 
 $$J^m(\alpha):=\{\tau:\rho_m(\tau)\in(\tan\alpha\frac \pi2,A^m]\}.$$ 
 If $A_m\leq \tan\alpha\frac\pi2$, then $J^m(\alpha)$ is an empty set. If $A_m>\tan\alpha\frac\pi2$ holds, since  $\rho_m(s)<\tan\alpha\frac\pi2$, there exists the minimum $\ov s>s$ such that $\rho_m(\ov s)=\tan\alpha\frac\pi2$\,, as well, being $\rho_m(t)<\tan\alpha\frac\pi2$, there exists the maximum $\ov t<t$ such that $\rho_m(\ov t)=\tan\alpha\frac\pi2$\,.    
Thus, if $J^m(\alpha)$ is a non-empty set, by the regularity of $\rho_m(t)$, we get that $J^m(\alpha)$ is at most the union of a sequence of open interval $(s_h,t_h)$ such that $\rho_m(s_h)=\rho_m(t_h)=\tan\alpha\frac\pi2$. We justify the claim. 
The set $J^m(\alpha)$ is an open set, hence it is at most the countable union of maximal intervals $(s_h,t_h)$.   We set  $E^m:=(s, t)-\ov{{\underset{h\in\N}\cup}(s_h,t_h)}$. \par For all $\tau\in E^m$ we have $\rho_m(\tau)\leq \tan\alpha\frac\pi2$, thus, by continuity of $\rho_m$, we get  $\rho_m(s_h) =\tan\alpha\frac\pi2=\rho_m(t_h)$ for all $h\in\N$.    
For the measure of $J^m(\alpha)$ we get
\be\label{MJ}|J^m(\alpha)|\tan\alpha\frac\pi2 \leq \intl{J^m(\alpha)}\rho_m(\tau)d\tau<\frac12\dm v(s)\dm_2^2+\intll st(f,v)d\tau\,,\ee 
where we took the energy relation \rf{ER} into account and the strong convergence of the right-hand side too. Estimate \rf{MJ} leads to \rf{LMJ}. Recalling the definition of $p(\alpha,\rho_m(t))$, we have 
\be\label{DT}\frac d{d\tau}p(\alpha,\rho_m(\tau))=\left\{\ba{ll}0&\mbox{a.e. in }\tau\in E^m\,,
\VS\frac{-2}{(1-\alpha)\pi}\frac{\dot\rho_m(\tau)}{1+(\rho_m(\tau))^2}&\mbox{for all  }\tau\in J^m(\alpha)\,,\ea\right.\ee 
where we took into account that, for all $\alpha\in(0,1)$, function $p$ is  Lipschitz's function in $\rho_m$, and $\rho_m(t)$ is  a  regular function in $t$. Hence,  we get  $p(\alpha,\rho_m(t))$ is Lipschitz's function with respect to $t$.
We multiply equation \rf{ER} for $p(\alpha,\rho_m(\tau))$, with $\alpha>\alpha_1$, and we integrate by parts on $(s,t)$:
$$\dm v^m(t)\dm_2^2+2\intll st p(\alpha,\rho_m(\tau))\dm\n v^m(\tau)\dm_2^2d\tau+\frac{2M(t,s,m,\alpha)}{(1-\alpha)\pi}=\dm v^m(s)\dm_2^2+\intll st(f,v^m)p(\alpha,\rho_m(\tau))d\tau\,,$$
where we set
$$M(t,s,m,\alpha):= \intl{J^m(\alpha)}\frac{\dm v^m(\tau)\dm_2^2}{1+\rho_m^2(\tau)}\hskip0.03cm{\dot\rho_m(\tau)}d\tau$$
where we took \rf{MV} and definition of $p$ into account. Letting $m\to\infty$ and $\alpha\to1$, by virtue of the pointwise convergence in $s$ and in $t$, and Lemma\,\ref{WC}, we arrive at
\be\label{SLM-I}\dm v(t)\dm_2^2+2 \intll st\dm\n v(\tau)\dm_2^2d\tau  -\dm v(s)\dm_2^2-\intll st(f,v)d\tau=-M(s,t)\,,\ee
where we set 
$$M(s,t):=\lim_{\alpha\to1^-}\overline{\lim_m} \frac2{(1-\alpha)\pi}\intl{J^m(\alpha)}\frac{\dm v^m(\tau)\dm_2^2}{1+\rho_m(\tau)^2}\dot\rho_m(\tau)d\tau\,.$$
 Recalling the properties of $J^m(\alpha)$, for all $\alpha$ and $m$, integrating by parts, we get 
 $$\ba{ll}\displ \intl{J^m(\alpha)}\!\!\frac{\dm v^m\dm_2^2}{1+\rho_m^2}\dot\rho_md\tau\hskip-0.2cm\!&\displ=\mbox{\large$\underset{h\in\N(\alpha,m)}\sum$}\intll{s_h}{t_h}\frac{\dm v^m\dm_2^2}{1+\rho_m^2}\dot\rho_md\tau\\&\displ=\frac{\tan\alpha\frac\pi2}{1+\tan^2\!\alpha\frac\pi2\hskip-0.1cm}\hskip0.1cm\mbox{\large$\underset{h\in\N(\alpha,m)}\sum$}\Big[ \dm v^m(t_h)\dm_2^2-\dm v^m(s_h)\dm_2^2\Big]+\mbox{\large$\underset{h\in\N(\alpha,m)}\sum$}\intll{s_h}{t_h}\frac{2\rho_m^2}{1+\rho_m^2}d\tau\\&\displ\hskip2cm
-2\sum_{h\in\N(\alpha,m)}\int\limits_{s_h}^{t_h}\frac{\rho_m(f,v_m)}{1+\rho_m^2}\,d\tau+2\mbox{\large$\underset{h\in\N(\alpha,m)}\sum$}\intll{s_h}{t_h}\frac{\dm v^m\dm_2^2\rho_m^2}{(1+\rho_m^2)^2\hskip-0.1cm}\hskip0.1cm\dot\rho_md\tau\,.\ea$$
Hence we arrive at
\be\label{SLM-I}\ba{l}\displ\mbox{\large$\underset{h\in\N(\alpha,m)}\sum$}\intll{s_h}{t_h}\frac{\dm v^m\dm_2^2}{1+\rho_m^2}\dot\rho_md\tau-2\mbox{\large$\underset{h\in\N(\alpha,m)}\sum$}\intll{s_h}{t_h}\frac{\dm v^m\dm_2^2\rho_m^2}{(1+\rho_m^2)^2\hskip-0.1cm}\hskip0.1cm\dot\rho_md\tau+\sum_{h\in\N(\alpha,m)}\int\limits_{s_h}^{t_h}\frac{\rho_m(f,v_m)}{1+\rho_m^2}\,d\tau\\
\hskip3cm\displ=\frac{\tan\alpha\frac\pi2}{1+\tan^2\!\alpha\frac\pi2\hskip-0.1cm}\hskip0.1cm\mbox{\large$\underset{h\in\N(\alpha,m)}\sum$}\Big[ \dm v^m(t_h)\dm_2^2-\dm v^m(s_h)\dm_2^2\Big]+\mbox{\large$\underset{h\in\N(\alpha,m)}\sum$}\intll{s_h}{t_h}\frac{2\rho_m^2}{1+\rho_m^2}d\tau\,. \ea\ee 
We estimate the last integral. 
Let be
\be\label{Jtilde}\widetilde J(\alpha):=\limsup_{m\to\infty}J^m(\alpha)=\bigcap_{j=0}^\infty\bigcup_{m=j}^\infty J^m(\alpha).\ee
It results that
$$\tau\in \widetilde J(\alpha)\iff \exists\, m_k\to\infty \mbox{ s.t. } \tau\in J^{m_k}(\alpha)\ \forall k\in\N\iff\limsup_{m\to\infty}\chi_{J^m(\alpha)}(\tau)=1.$$
hence, if $\tau\in \widetilde J(\alpha)\cap\mathcal T$ we get that
\be\label{rhoonJalpha}\rho(\tau)=\lim_{k\to\infty}\rho_{m_k}(\tau)\ge\tan\frac{\alpha\pi}2.\ee
On the complement of the set $\mathcal T$ we can set $\rho=0$, since the value on a null measure set does not change the estimates.
Since $0\le\chi_{J^m(\alpha)}\frac{\rho_m^2}{1+\rho_m^2}\le1$, by Fatou's lemma, it follows that
\begin{equation*}\begin{split}\frac1{1-\alpha}\limsup_m\int\limits_{J^m(\alpha)}\frac{\rho_m^2}{1+\rho_m^2}\,d\tau=\frac1{1-\alpha}\limsup_m\int\limits_s^t\chi_{J^m(\alpha)}\frac{\rho_m^2}{1+\rho_m^2}\,d\tau\\
\le\frac1{1-\alpha}\int\limits_s^t\limsup_m\chi_{J^m(\alpha)}\frac{\rho_m^2}{1+\rho_m^2}\,d\tau=\frac1{1-\alpha}\int\limits_s^t\chi_{\widetilde J(\alpha)}\frac{\rho^2}{1+\rho^2}\,d\tau=\frac1{1-\alpha}\int\limits_{\widetilde J(\alpha)}\frac{\rho^2}{1+\rho^2}\,d\tau\\
\le\frac1{1-\alpha}\frac1{\tan\frac{\alpha\pi}2}\int\limits_{\widetilde J(\alpha)}\rho(\tau)\,d\tau
\end{split}\end{equation*}
Since $\rho\in L^1$ and, by \eqref{rhoonJalpha}, 
\be\label{measJtilde}
\left|\widetilde J(\alpha)\right|\le \frac{\|\rho\|_1}{\tan\frac{\alpha\pi}2}
\ee
 the last integral vanishes as $\alpha$ tends to $1^-$. Moreover
\be\label{LD}\lim_{\alpha\to1^-}\frac1{(1-\alpha)\tan\alpha\frac\pi2}=\frac\pi2\,\ee
hence 
\be\label{rho2frac1+rho2}\lim_{\alpha\to1^-}\overline{\lim_m}\frac1{1-\alpha}\int\limits_{J^m(\alpha)}\frac{\rho_m^2}{1+\rho_m^2}\,d\tau=0.\ee
Concerning the force term we have
\begin{equation*}\begin{split}\left|\int\limits_{J^m(\alpha)}\frac{\rho_m(f,v^m)}{1+\rho_m^2}\,d\tau\right|
\le\left(\int\limits_{J^m(\alpha)}\frac{\rho_m^2}{(1+\rho_m^2)^2}\,d\tau\right)^{\frac12}\left(\int\limits_{J^m(\alpha)}\|f(\tau)\|_2^2\|v^m(\tau)\|_2^2\,d\tau\right)^{\frac12}\\
\le\left(\frac1{1+(\tan\frac{\alpha\pi}2)^2}\right)^{\frac12}\left(\int\limits_{J^m(\alpha)}\frac{\rho_m^2}{1+\rho_m^2}\,d\tau\right)^{\frac12}\sup_{m,t}\|v^m(t)\|_2\left(\int\limits_{J^m(\alpha)}\|f(\tau)\|_2^2\,d\tau\right)^{\frac12}\\
\le\frac C{\tan\frac{\alpha\pi}2}\left(\frac{(\tan\frac{\alpha\pi}2)^2}{1+(\tan\frac{\alpha\pi}2)^2}\right)^{\frac12}|J^m(\alpha)|^{\frac12}\le\frac C{\tan\frac{\alpha\pi}2}\left(\frac c{\tan\frac{\alpha\pi}2}\right)^{\frac12}.
\end{split}\end{equation*}
It follows that
\be\label{force}\lim_{\alpha\to1^-}\frac1{1-\alpha}\overline{\lim_m}\int\limits_{J^m(\alpha)}\frac{\rho_m(f,v^m)}{1+\rho_m^2}\,d\tau=0.\ee  
Using algebraic manipulation we obtain the following relation:
$$\ba{l}\displ\mbox{\large$\underset{h\in\N(\alpha,m)}\sum$}\intll{s_h}{t_h}\frac{\dm v^m\dm_2^2}{1+\rho_m^2}\dot\rho_md\tau-2\mbox{\large$\underset{h\in\N(\alpha,m)}\sum$}\intll{s_h}{t_h}\frac{\dm v^m\dm_2^2\rho_m^2}{(1+\rho_m^2)^2\hskip-0.1cm}\hskip0.1cm\dot\rho_md\tau\\
\hskip2cm\displ =-\mbox{\large$\underset{h\in\N(\alpha,m)}\sum$}\intll{s_h}{t_h}\frac{\dm v^m\dm_2^2}{1+\rho_m^2}\dot\rho_md\tau +2\mbox{\large$\underset{h\in\N(\alpha,m)}\sum$}\intll{s_h}{t_h}\frac{\dm v^m\dm_2^2}{(1+\rho_m^2)^2\hskip-0.1cm}\hskip0.1cm\dot\rho_md\tau\,.\ea$$
Substituting the above relation in equation \eqref{SLM-I} we get
\begin{equation}\label{algebraic}\begin{split}\displ-\mbox{\large$\underset{h\in\N(\alpha,m)}\sum$}\intll{s_h}{t_h}\frac{\dm v^m\dm_2^2}{1+\rho_m^2}\dot\rho_md\tau +2\mbox{\large$\underset{h\in\N(\alpha,m)}\sum$}\intll{s_h}{t_h}\frac{\dm v^m\dm_2^2}{(1+\rho_m^2)^2\hskip-0.1cm}\hskip0.1cm\dot\rho_md\tau+\sum_{h\in\N(\alpha,m)}\int\limits_{s_h}^{t_h}\frac{\rho_m(f,v_m)}{1+\rho_m^2}\,d\tau\\
\displ=\frac{\tan\alpha\frac\pi2}{1+\tan^2\!\alpha\frac\pi2\hskip-0.1cm}\hskip0.1cm\mbox{\large$\underset{h\in\N(\alpha,m)}\sum$}\Big[ \dm v^m(t_h)\dm_2^2-\dm v^m(s_h)\dm_2^2\Big]+\mbox{\large$\underset{h\in\N(\alpha,m)}\sum$}\intll{s_h}{t_h}\frac{2\rho_m^2}{1+\rho_m^2}d\tau\end{split}\end{equation}
At last we estimate the integral
\begin{equation}\label{last}\begin{split}\left|\int\limits_{J^m(\alpha)}\frac{\dm v^m\dm_2^2}{(1+\rho_m^2)^2\hskip-0.1cm}\hskip0.1cm\dot\rho_md\tau\right|\le\int\limits_{J^m(\alpha)}\frac{\dm v^m\dm_2^2}{(1+\rho_m^2)^2\hskip-0.1cm}\hskip0.1cm|\dot\rho_m|\,d\tau\le\sup_{t,m}\|v^m(t)\|_2^2\int\limits_{J^m(\alpha)}\frac{|\dot\rho_m|}{(1+\rho_m^2)^2\hskip-0.1cm}\,d\tau\\
\le c\,\frac1{1+(\tan\frac{\alpha\pi}2)^2}\int\limits_{J^m(\alpha)}\frac{|\dot\rho_m|}{1+\rho_m^2\hskip-0.1cm}\,d\tau\le\frac{2c}{1+(\tan\frac{\alpha\pi}2)^2}\int\limits_{J^m(\alpha)}\frac{|\dot\rho_m|}{(1+\rho_m)^2\hskip-0.1cm}\,d\tau\le\frac{2cM}{1+(\tan\frac{\alpha\pi}2)^2}
\end{split}\end{equation}
where the last inequality follows by Lemma \ref{CL}. Hence, by \eqref{LD}, we get
$$\lim_{\alpha\to1^-}\limsup_m\left|\frac2{1-\alpha}\int\limits_{J^m(\alpha)}\frac{\dm v^m\dm_2^2}{(1+\rho_m^2)^2\hskip-0.1cm}\hskip0.1cm\dot\rho_md\tau\right|
\le\lim_{\alpha\to1^-}\frac2{1-\alpha}\frac{2cM}{1+(\tan\frac{\alpha\pi}2)^2}=0.
$$
Multiplying equation \eqref{algebraic} by $\frac2{(1-\alpha)\pi}$ and passing to the limit using \eqref{last}, \eqref{force} and \eqref{rho2frac1+rho2}, we get
$$-M(s,t)=\lim_{\alpha\to1^-}\overline{\lim_m}\mbox{\large$\underset{h\in\N(\alpha,m)}\sum$}\Big[ \dm v^m(t_h)\dm_2^2-\dm v^m(s_h)\dm_2^2\Big].$$
By equation \eqref{ER} we get
$$\mbox{\large$\underset{h\in\N(\alpha,m)}\sum$}\Big[ \dm v^m(t_h)\dm_2^2-\dm v^m(s_h)\dm_2^2\Big]=-2\int\limits_{J^m(\alpha)}\|\nabla v^m(\tau)\|_2^2\,d\tau+\int\limits_{J^m(\alpha)}(f,v^m)\,d\tau.$$
Let us consider the last integral. Since $\left|\left(f(\tau),v^m(\tau)\right)\right|\le\|f(\tau)\|_2\|v^m(\tau)\|_2\le c\|f(\tau)\|_2$ we can apply the Fatou's lemma to get
\begin{equation*}\begin{split}
\limsup_m\left|\int\limits_{J^m(\alpha)}(f,v^m)\,d\tau\right|
=\limsup_m\left|\int\limits_s^t\chi_{J^m(\alpha)}(f,v^m)\,d\tau\right|
\le\int\limits_s^t\limsup_m\left|\chi_{J^m(\alpha)}(f,v^m)\right|\,d\tau\\
\le c\int\limits_s^t\|f\|_2\limsup_m\chi_{J^m(\alpha)}\,d\tau
\le c\int\limits_s^t\|f\|_2\chi_{\widetilde J(\alpha)}\,d\tau
\end{split}\end{equation*}
with $\widetilde J(\alpha)$ defined in \eqref{Jtilde}. Since $\|f(\tau)\|_2$ is summable, considering \eqref{measJtilde}, we get
$$\lim_{\alpha\to1^-}\limsup_m\left|\int\limits_{J^m(\alpha)}(f,v^m)\,d\tau\right|=0$$
and this completes the proof in the case of $s,t\in\mathcal T$. In order to complete the proof of the theorem, we limit ourselves to  remark that, letting $s\to0$, the left-hand side tends to values in $0$, in particular on any sequence $\{s_k\}\subset \mathcal T$  letting to 0, and as a consequence the limit on $\{s_k\}$ of the right hand side is well posed.    
\ep

{\bf Acknowledgements -} The research activity of F.C. and P.M. is performed under the
auspices of   GNFM-INdAM, and the research activity of C.R.G. is performed under the auspices of GNAMPA-INDAM. \newline The research activity of F.C has been supported by the Program (Vanvitelli per la Ricerca: VALERE) 2019 financed by the University of Campania ``L. Vanvitelli''.\newline The research activity of C.R.G. is partially supported by PRIN 2020 ``Nonlinear evolution PDEs, fluid dynamics and transport equations: theoretical foundations and applications.''
\vskip0.1cm 
{\bf Declarations\vskip0.1cm\par Conflict of interests} The authors declare that they have no conflict of interest.


{\small
\begin{thebibliography}{99}\bibitem{ABC}D. Albritton, E. Bru\é and M. Colombo, {\it Non-uniqueness of Leray solutions of the forced Navier-Stokes equations,} arXiv:2112.031116v1, (2021).
\bibitem{BV}
T. Buckmaster and V. Vicol, \emph{Nonuniqueness of weak solutions to the
  {N}avier-{S}tokes equation}, Ann. of Math. (2) \textbf{189} (2019), no.~1,
  101--144. 
\bibitem{CKN}
L. Caffarelli, R.~Kohn, and L.~Nirenberg, \emph{Partial regularity of suitable
  weak solutions of the {N}avier-{S}tokes equations}, Comm. Pure Appl. Math.
  \textbf{35} (1982), no.~6, 771--831. 
\bibitem{CGM-I}F. Crispo, C.R. Grisanti and P. Maremonti, {\it Some new properties of a suitable weak solution to the Navier- Stokes equations}, in Waves in Flows: The 2018 Prague-Sum Workshop Lectures, series: Lecture Notes in Mathematical Fluids Mechanics, editors: G.P.Galdi, T. Bodnar, S. Necasova, Birkhauser.
\bibitem{CGM-II}F. Crispo, C.R. Grisanti and P. Maremonti, {\it Navier-Stokes equations: an analysis of a possible gap to achieve the energy equality}, Ricerche di Matematica, {\bf 70} (2021) 235-249, https://doi.org/10.1007/s11587-020-00525-5
\bibitem{GPG}
G.P. Galdi, \emph{On the relation between very weak and Leray-Hopf solutions to Navier-Stokes equations},  Proc. Amer. Math. Soc. 147 (2019), 5349-5359.
\bibitem{L}
J. Leray, \emph{Sur le mouvement d'un liquide visqueux emplissant l'espace},
  Acta Math. \textbf{63} (1934), no.~1, 193--248. 
\bibitem{MRIII}
P. Maremonti, \emph{Some interpolation inequalities involving {S}tokes
  operator and first order derivatives}, Ann. Mat. Pura Appl. (4) \textbf{175}
  (1998), 59--91. 
\bibitem{Mi}
P. Maremonti, \emph{A note on {P}rodi-{S}errin conditions for the regularity of a
  weak solution to the {N}avier-{S}tokes equations}, J. Math. Fluid Mech.
  \textbf{20} (2018), no.~2, 379--392. 
\bibitem{MRIV}
P. Maremonti, \emph{On an interpolation inequality involving the {S}tokes operator},
  Mathematical analysis in fluid mechanics---selected recent results, Contemp.
  Math., vol. 710, Amer. Math. Soc., Providence, RI, 2018, pp.~203--209.
\bibitem{Sh} V. Scheffer, {\it Hausdorff measure and the Navier-Stokes equations}, Comm. Math. Phys., {\bf 55} (1977),  97--112.
\end{thebibliography}}
\end{document}